\documentclass[12pt,reqno]{amsart}

\newtheorem{thm}{Theorem}
\newtheorem{prop}[thm]{Proposition}
\newtheorem{cor}[thm]{Corollary}
\newtheorem{lem}[thm]{Lemma}
\theoremstyle{definition}
\newtheorem*{rem}{Remark}
\newtheorem{ex}{Example}

\providecommand{\RR}{\mathbb{R}}
\providecommand{\CC}{\mathbb{C}}
\providecommand{\QQ}{\mathbb{Q}}

\DeclareMathOperator{\sym}{Sym}

\providecommand{\eps}{\epsilon}
\DeclareMathOperator{\lt}{lt}
\DeclareMathOperator{\id}{id}
\DeclareMathOperator{\tr}{tr}
\DeclareMathOperator{\diag}{diag}
\DeclareMathOperator{\cha}{char}

\begin{document}

\title{Formally real involutions on central simple algebras}{}

\author{J. Cimpri\v c}



\date{August 10th 2006, revised January 28th 2007}

\address{Jaka Cimpri\v c, University of Ljubljana, Faculty of Math. and Phys.,
Dept. of Math., Jadranska 19, SI-1000 Ljubljana, Slovenija. 
E-mail: cimpric@fmf.uni-lj.si. www page: http://www.fmf.uni-lj.si/ $\!\!\sim$cimpric.}

\begin{abstract}
An involution $\#$ on an associative ring $R$ is \textit{formally real} if 
a sum of nonzero elements of the form $r^\# r$ where $r \in R$ is nonzero.
Suppose that $R$ is a central simple algebra (i.e. $R=M_n(D)$ for some integer $n$
and central division algebra $D$) and $\#$ is an involution on $R$ of the form 
$r^\# = a^{-1} r^\ast a$, where $\ast$ is some transpose involution on $R$ 
and $a$ is an invertible matrix such that $a^\ast=\pm a$.
In section 1 we characterize formal reality of $\#$ in terms of $a$ and $\ast|_D$.
In later sections we apply this result to the study of formal reality
of involutions on crossed product division algebras. We can characterize
involutions on $D=(K/F,\Phi)$ that extend to a formally real involution 
on the split algebra $D \otimes_F K \cong M_n(K)$. Every such involution 
is formally real but we show that there exist formally real involutions 
on $D$ which are not of this form. In particular, there exists
a formally real involution $\#$ for which the hermitian trace form
$x \mapsto \tr(x^\#x)$ is not positive semidefinite.
\end{abstract}

\maketitle

\thispagestyle{empty}
\section{$\eps$-hermitian cones on central simple algebras}

We say that an involution $\ast$ on a central simple algebra $R$ is \textit{formally real}
if any finite sum of nonzero elements of the form $rr^\ast$ where $r \in R$ is nonzero. 
In this section we introduce our main technical tool for  the study of formally real involutions
- the notion of an $\eps$-hermitian cone. The precise relationship between
$\eps$-hermitian cones and formally real involutions is explained by Corollary \ref{firstcorr}.

Recall that a central simple algebra is a full matrix ring over a central division algebra.
Let $R$ be a central simple $F$-algebra with involution $\ast$ and $\eps \in F$ 
such that $\eps \eps^\ast=1$.  An element $a \in R$ is $\eps$-hermitian if $\eps a^\ast=a$. 
The set of all $\eps$-hermitian elements in $R$ will be denoted by $S_\eps(R)$. 
A subset $M$ of $S_\eps(R)$ such that $M+M \subseteq M$, $aMa^\ast \subseteq M$ for every $a \in R$
and $M \cap -M = \{0\}$ will be called an $\eps$-\textit{hermitian cone} on $R$.

\begin{thm}
\label{main}
Let $D$ be a central division $F$-algebra with involution $\ast$. For every $\eps \in F$ such that
$\eps \eps^\ast=1$ and for every integer $n$ there exists a one-to-one correspondence between
\begin{itemize}
\item $\eps$-hermitian cones on $D$ and 
\item $\eps$-hermitian cones on $M_n(D)$ (with involution $[x_{ij}]^\ast=[x_{ji}^\ast]$).
\end{itemize}
\end{thm}

We will need the following well-known \lq\lq Diagonalization Theorem".

\begin{prop}
\label{diag}
Every matrix $A \in S_\eps(R)$ is congruent to a direct sum of  matrices
of the form $[a]$ or $\left[ \begin{array}{cc}0 & b \\ \eps b^\ast & 0 \end{array} \right]$ 
where $a \in S_\eps(D)$ and $b \in D$. Moreover, if either $\eps \ne -1$ or $\ast |_D \ne \id$
then every matrix of the form
 $\left[ \begin{array}{cc}0 & b \\ \eps b^\ast & 0 \end{array} \right]$ 
is congruent to a diagonal matrix.
\end{prop}

\begin{proof} The first claim is proved by induction on $n$ using the identity
\[
\left[ \begin{array}{cc} I & 0 \\ -B_{21}B_{11}^{-1} & I \end{array} \right]
\left[ \begin{array}{cc} B_{11} & B_{12} \\ B_{21} & B_{22} \end{array} \right]
\left[ \begin{array}{cc} I & -B_{11}^{-1} B_{12} \\ 0 & I \end{array} \right] =
\left[ \begin{array}{cc} B_{11} & 0 \\ 0 & \tilde{B}_{11} \end{array} \right]
\]
where $\tilde{B}_{11} = B_{22}-B_{21}B_{11}^{-1}B_{12}$ is 
the Schur complement of $B_{11}$. The second claim is just a short computation.
\end{proof}

Now, we can prove the theorem.

\begin{proof}
Let us start with the case $\eps=-1$ and $\ast|_D = \id$. In this case $D=F$ and $2S_\eps(D)=0$.
Let $M$ be an $\eps$-hermitian cone on $M_n(D)$. Pick any $C \in M$. By Proposition \ref{diag},
there exists an invertible matrix $P$ such that 
\[
P^\ast C P = \bigoplus_{i=1}^r [a_i] \oplus \bigoplus_{j=1}^s 
\left[ \begin{array}{cc}0 & b_j \\ -b_j & 0 \end{array} \right],
\]
where $a_i \in S_\eps(D)$ and $b_j \in D$. Note that the matrix
\[
Q = \bigoplus_{i=1}^r [1] \oplus \bigoplus_{j=1}^s 
\left[ \begin{array}{cc}0 & 1 \\ 1 & 0 \end{array} \right]
\]
is invertible and $Q^\ast P^\ast C P Q=-P^\ast C  P$. The latter follows from
\[
[a_i] = -[a_i] \quad \mbox{and} \quad
\left[ \begin{array}{cc}0 & 1 \\ 1 & 0 \end{array} \right]
\left[ \begin{array}{cc}0 & b_j \\ -b_j & 0 \end{array} \right]
\left[ \begin{array}{cc}0 & 1 \\ 1 & 0 \end{array} \right] =
- \left[ \begin{array}{cc}0 & b_j \\ -b_j & 0 \end{array} \right].
\]
It follows that $P^\ast C P \in M \cap -M =\{0\}$, so that $C=0$. Hence
$\{0\}$ is the only $\eps$-hermitian cone on $M_n(D)$. For $n=1$ we get
that $\{0\}$ is also the only $\eps$-hermitian cone on $D$.

From now on we assume that either $\eps \ne -1$ or $\ast|_D \ne \id$. Therefore,
every $\eps$-hermitian matrix is congruent to a diagonal matrix.

For every $\eps$-hermitian cone $N$ of $D$ write
\[ 
F(N) = \{A \in S_\eps(M_n(D)) \vert \ x A x^\ast \in N \text{ for every } x \in D^n\}.
\]
We claim that $F(N)$ is an $\eps$-hermitian cone on $M_n(D)$. If $A,B \in F(N)$, then
$x(A+B)x^\ast = xAx^\ast+xBx^\ast \in N$ for every $x \in D^n$. Hence $A+B \in F(N)$.
If $A \in F(N)$ and $B \in M_n(D)$, then $xBAB^\ast x^\ast =(xB)A(xB)^\ast \in N$
for every $x \in D^n$. Hence $BAB^\ast \in F(N)$. If $N$ is proper, then $F(N)$ is also
proper. Namely, if $A \in F(N) \cap -F(N)$, then $xAx^\ast \in N \cap -N = \{0\}$ for 
every $x \in D^n$. Let $A'$ be the diagonal matrix congruent to $A$. Then $xA'x^\ast=0$
for every $x \in D^n$, so that $A'=0$. Hence, $A=0$.

For every $\eps$-hermitian cone $M$ on $M_n(D)$ write 
\[
G(M) = \{c \in S_\eps(D) \vert \ cE_{11} \in M\}.
\]
We claim that $G(M)$ is an $\eps$-hermitian cone on $D$.
If $a,b \in G(M)$, then $(a+b)E_{11}=aE_{11}+bE_{11} \in M$, hence $a+b \in G(M)$.
If $a \in G(M)$ and $d \in D$ then $(dad^\ast)E_{11}=(dE_{11})(aE_{11})(dE_{11})^\ast \in M$,
hence $dad^\ast \in G(M)$. If $M$ is proper, then $G(M)$ is also proper. Namely, if 
$a \in G(M) \cap -G(M)$, then $aE_{11} \in M \cap -M = \{0\}$. Hence $aE_{11}=0$.

We claim that for every $\eps$-hermitian cone $N$ on $D$ we have that $G(F(N)) = N$.
If $a \in G(F(N))$, then $aE_{11} \in F(N)$. Then $x(aE_{11})x^\ast \in N$ for every
$x \in D^n$. In particular, for $x = (1,0,\ldots,0)$, we get that $a \in N$.
To prove the opposite inclusion pick any $a \in N$. Then $aE_{11} \in F(N)$
since $x(aE_{11})x^\ast = x_1 a x_1^\ast \in N$ for every $x \in D^n$.
It follows that $a \in G(F(N))$.

We claim that for every $\eps$-hermitian cone $M$ on $M_n(D)$ we have that $F(G(M)) = M$.
Since every $\eps$-hermitian matrix is congruent to a diagonal matrix, it suffices to show that
a diagonal matrix belongs to $F(G(M))$ if and only if it belongs to $M$.
Pick a diagonal matrix $A \in F(G(M))$. Then for every $x \in D^n$, $xAx^\ast \in G(M)$.
In particular $a_{11},\ldots,a_{nn} \in G(M)$. It follows that 
$a_{11}E_{11},\ldots,a_{nn}E_{11} \in M$. Hence 
$A = \sum_j P_{1j}(a_{jj}E_{11})P_{1j}^\ast \in M$.
To prove the opposite inclusion, pick $A \in M$. It follows that 
$a_{jj}E_{11} = E_{1j}AE_{1j}^\ast \in M$ for every $j$.
Hence $a_{11},\ldots,a_{nn} \in G(M)$. It follows that
$xAx^\ast = \sum_j x_j a_{jj} x_j^\ast \in G(M)$ for every
$x \in D^n$. Hence $A \in F(G(N))$.
\end{proof}

Let $R=M_n(D)$ be a central simple $F$-algebra with $\cha F \ne 2$. The following is a summary
of \cite{invo}, Theorem 3.1 and Proposition 2.20: For every involution $\#$ on $M_n(D)$ there exists  
an involution $\ast$ on $D$ such that  $\# |_F = \ast|_F$ and an invertible matrix $A \in M_n(D)$ 
such that $A^\ast=\pm A$ and $X^\# = A^{-1} X^\ast A$ for every matrix $X \in M_n(D)$ 
(where $[x_{ij}]^\ast=[x_{ji}^\ast]$). If $\ast|_F \ne \id$ we can assume $A^\ast=A$.

Our next result extends Theorem \ref{main} to arbitrary involutions.

\begin{thm}
\label{arbitrary}
Let $D$ be a central division $F$-algebra with involution $\ast$
and $\eta \in F$ such that $\eta\eta^\ast=1$. Let $n$ be an integer
and $A \in S_\eta(M_n(D))$ an invertible matrix. Then $X^\# = A^{-1}X^*A$ 
is also an involution on $M_n(D)$. For every $\eps \in F$ such that
$\eps \eps^\ast=1$ there exists a one-to-one correspondence between:
\begin{enumerate}
\item $\eps$-hermitian cones on $(M_n(D),\#)$,
\item $\eps \eta$-hermitian cones on $(M_n(D),\ast)$,
\item $\eps \eta$-hermitian cones on $D$.
\end{enumerate}
\end{thm}

\begin{proof}
Let $\phi \colon M_n(D) \to M_n(D)$ be the mapping defined by $\phi(X)=AX$.
A short computation shows that $\phi$ induces a one-to-one correspondence
between the set $S_\eps(M_n(D),\#)$ and the set $S_{\eps \eta}(M_n(D),*)$.
Moreover $\phi$ is additive and for every $X \in S_\eps(M_n(D),\#)$ and $Y \in M_n(D)$
we have that $\phi(Y^\# X Y)=Y^\ast \phi(X) Y$. 
Therefore the mappings $M \to \phi(M)$ and $N \to \phi^{-1}(N)$
give a one-to-one correspondence between (1) and (2) for every $\eps$.
The mappings $F$ and $G$ from Theorem \ref{main} give a 
one-to-one correspondence between (2) and (3).
\end{proof}

As a corollary of Theorem \ref{arbitrary}, we obtain
a characterization of formally real involutions on central simple algebras.

\begin{cor}
\label{firstcorr}
Let $D,F,\ast,\eta,n,A,\#$ be as in Theorem \ref{arbitrary}. If either $\eta \ne -1$ or $\ast|_D \ne \id$,
then the following assertions are equivalent:
\begin{enumerate}
\item $\#$ is formally real,
\item there exists a $1$-hermitian cone on $(M_n(D),\#)$ which contains the identity matrix $I$,
\item there exists an $\eta$-hermitian cone on $(M_n(D),\ast)$ which contains the matrix $A$,
\item there exists an $\eta$-hermitian cone on $(D,\ast)$ which contains all elements in some diagonal
representation of $A$.
\end{enumerate}
If $\eta=-1$ and $\ast|_D = \id$ then $\#$ is not formally real.
\end{cor}

A $1$-hermitian cone containing $1$ will be called \textit{unital hermitian cone} in the sequel.
If $A$ has a nonzero hermitian square on its diagonal (e.g. $1$) then every $1$-hermitian cone
on $(M_n(D),\ast)$ which contains the matrix $A$ is a unital hermitian cone. If $D$ is a division
algebra admitting a unital hermitian cone then $\cha K =0$ for every subfield $K$ of $D$.

\section{Extensions of involutions from $D$ to $D \otimes K$.}
\label{secinvo}

Let $D$ be a central division $F$-algebra and $K$ a maximal subfield of $D$.
Every involution $\ast$ on $D$ such that $K^\ast \subseteq K$ extends to an involution
on $D \otimes_F K$. This is clear if $\ast$ is of the first kind (i.e. $\ast|_F =\id$).
If $\ast$ is of the second kind (i.e. $\ast|_F \ne \id$), one has to observe first that
$D \otimes_F K$ is isomorphic to $D \otimes_{F_0} K_0$ where $K_0 = S_1(K)=\sym(K)$
and $F_0 =S_1(F)=\sym(F)$ are the symmetric parts of $K$ and $F$. The aim of this section
is to give an explicit construction of the extension which will be used in later
sections. From now on we assume that $\cha K = 0$.

\begin{lem}
For every involution $\ast$ on $D$ there exists a maximal (i.e. self-centralizing) 
subfield of $D$ which is $\ast$-invariant.
\end{lem}

\begin{proof} We need the following claim:
\textit{A central simple algebra with involution in which every normal element is central is a field.}

Let $A$ be a central simple algebra in which every normal element is central. 
Pick any element $a \in A$. The element $\alpha = a+a^\ast$ is symmetric, hence normal. 
By the assumption $\alpha \in Z(A)$. It follows that $aa^\ast = a(\alpha-a) = (\alpha-a)a = a^\ast a$.
Since $a$ is normal, it is central by the assumption. Hence $A = Z(A)$.

Suppose now that $L$ is a $\ast$-subfield of a division algebra $D$ which is not contained
in any other $\ast$-subfield. Its centralizer $A = C_D(L)$ is also $\ast$-invariant and it contains $L$. 
By the Double Centralizer Theorem, $A$ is a simple algebra with $Z(A)=L$. If $A \ne L$,
then $A$ contains a noncentral normal element $d$ by the claim.  Then $L(d,d^\ast)$ is also a 
$\ast$-subfield of $D$ properly containing $L$, contrary to the choice of $L$. Therefore $A=L$.
\end{proof}

\begin{lem}
\label{anti}
Every finite extension of fields with involution is generated
by either a symmetric or an antisymmetric element.
(We assume that both fields have characteristic zero.)
\end{lem}

\begin{proof}
Let $K/F$ be a finite extension of fields with involution. 
This means that $K$ is a field with involution $\ast$, 
$F^\ast \subseteq F$ and $K/F$ is a finite field extension.
Write $F_0 = \sym(F)$ and $K_0 = \sym(K)$. By the Primitive Element Theorem,
there exists an element $\theta$ such that $K = F_0(\theta)$. 
If $K=K_0$ then $\theta$ is symmetric and we are done. If $K \ne K_0$, 
then $K = F_0(\theta-\theta^\ast,(\theta-\theta^\ast)(\theta+\theta^\ast))$
and both generators are antisymmetric. By the Primitive Element Theorem, there exists
an $F_0$-linear combination $\theta'$ of them such that $K=F_0(\theta')$. Hence,
$\theta'$ is antisymmetric and $K=F(\theta')$. It is also interesting to note that
if $F \ne F_0$ then $K/F$ is generated by a symmetric element. 
Namely, $K=F(k \theta')$, where $k$ is an antisymmetric generator of $F/F_0$.
\end{proof}

\begin{prop}
Let $K$ be a maximal $\ast$-subfield of a division algebra $D$.
There exists a unital, hermitian and $K-K$ bilinear mapping $f \colon D \to K$.
\end{prop}

\begin{proof}
Let $D$ be a division algebra with center $F$ and $K$ a maximal $\ast$-subfield in $D$.
Pick a symmetric or antisymmetric element $x \in K$ such that $K = F_0(x)$
and let $\chi(t) = t^n+a_{n-1}t^{n-1}+\ldots+a_1 t+a_0$ be its minimal $F_0$-polynomial.
Write
\[
\begin{array}{l}
y_0 = a_0 x^{n-1}, \\
y_1 = a_0 x^{n-2}+a_1 x^{n-1}, \\
y_2 = a_0 x^{n-3}+a_1 x^{n-2}+a_2 x^{n-1}, \\
\vdots \\
y_{n-2} = a_0 x+ a_1 x^2 +a_2 x^3+\ldots+a_{n-2} x^{n-1}, \\
y_{n-1} = a_0 + a_1 x+a_2 x^2+\ldots+a_{n-2}x^{n-2}+a_{n-1}x^{n-1} = -x^n.
\end{array}
\]
The mapping $f$ is defined by
\[
f(z)=\frac{-1}{\chi'(x) x^n} \sum_{i=0}^{n-1} x^i z y_i.
\]
The coefficient was chosen so that $f(1)=1$ as one can easily verify.
Clearly, $f$ is $K-K$ bilinear. A short computation shows that for every $z \in D$,
$xf(z) = f(z)x$. This relation implies that $f(z)$ commutes with all elements from $K$, 
hence it belongs to $K$. To prove that $f(z^\ast) = f(z)^\ast$ for every $z \in D$ one
has to distinguish the case when $x$ is symmetric from the case when $x$ is antisymmetric.
The symmetric case is easy. In the antisymmetric case we use the fact that $\chi(t)$ has only even powers.
\end{proof}

\begin{rem}
Note that $\tr = \tr_{K/F} \circ f$, where $\tr$ is the reduced trace. 
\end{rem}

Let $K$ be a maximal subfield of a division algebra $D$ and $e_1,\ldots,e_n$ a right $K$-basis of $D$.
Let $\lambda \colon D \to M_n(K)$ be the left regular representation defined by 
$a [e_1,\ldots,e_n] = [e_1,\ldots,e_n] \lambda(a)$ for $a \in D$ and let $j \colon K \to M_n(K)$
be the natural imbedding defined by $j(k)=kI$ where $I \in M_n(K)$ is the identity matrix.
Then the mapping $D \otimes_F K \to M_n(K)$ defined by $a \otimes k \mapsto \lambda(a)j(k)$
is an isomorphism.

\begin{prop}
\label{extend}
Setup from above. Let $\ast$ be an involution on $D$ which leaves $K$ invariant
and let $[a_{ij}]^\ast =[a_{ji}^\ast]$ be its extension to $M_n(K)$. Then the matrix 
$A = [f(e_i^\ast e_j)]_{i,j=1,\ldots,n} \in M_n(K)$ is nonsingular and hermitian (i.e. $A^\ast=A$) 
and the involution $\#$ on $M_n(K)$ defined by $X^\# = A^{-1}X^\ast A$ satisfies 
$\lambda(a^\ast)=\lambda(a)^\#$ for every $a \in D$.     
\end{prop} 

\begin{proof}
Since $f$ is hermitian, it follows that $A$ is hermitian. If $A$ is singular, then there exists a vector 
$v = (\alpha_1,\ldots,\alpha_n) \in K^n$ such that $Av=0$. Since $f$ is right $K$-linear, it follows that
$f(e_i^\ast (\sum_{j=1}^n(e_j \alpha_j))=0$ for every $i$. Since $f$ is left $K$-linear, it follows that
$f((\sum_{i=1}^n e_i \beta_i)^\ast (\sum_{j=1}^n e_j \alpha_j))=0$ for any $\beta_1,\ldots,\beta_n \in K$.
Since $D$ is a division algebra we can pick $\beta_i$ so that 
$(\sum_{i=1}^n e_i \beta_i)^\ast (\sum_{j=1}^n e_j \alpha_j)=1$. 
This is a contradiction with $f(1)=1$. Hence $A$ is nonsingular.
By the definition of $\lambda$, we get that 
\[
\lambda(a)^\ast 
\left[ \begin{array}{ccc} 
e_1^\ast e_1 & & e_1^\ast e_n\\ 
               & & \\
e_n^\ast e_1 & & e_n^\ast e_n
\end{array} \right] 
=
\left[ \begin{array}{c} e_1^\ast \\ \\ e_n^\ast \end{array} \right]
a^\ast 
\left[ \begin{array}{ccc} e_1 & & e_n \end{array} \right] 
=
\left[ \begin{array}{ccc} 
e_1^\ast e_1 & & e_1^\ast e_n\\ 
               & & \\
e_n^\ast e_1 & & e_n^\ast e_n
\end{array} \right]  
\lambda(a^\ast)
\]
for every $a \in D$. Applying $f$ to all elements of this identity we get
that $\lambda(a^\ast) = A^{-1} \lambda(a)^\ast A = \lambda(a)^\#$
for every $a \in D$.
\end{proof}

\begin{ex}
Let $D = \left( \frac{a, b}{F} \right)$ be a quaternion algebra. Recall that $D$
is an $F$-algebra with two generators $i$ and $j$, and three relations $i^2=a,j^2=b,ij=-ji$.
It is a division algebra if and only if the only solution in $F$ 
of the equation $ax^2+by^2=z^2$ is $x=y=z=0$.
If $z = \alpha+\beta i+\gamma j+\delta k$ then $f(z)=\alpha+\beta i$.
(Take $x=i,\chi(t)=t^2-a,y_0=-ai,y_1=-a,n=2$ in the definition of $f$.)
Let $\ast$ be an involution of $D$ defined by $i^\ast=i$ and $j^\ast=j$.
Write $k = ij$. Then $k^\ast=-k$ and $1,i,j,k$ is an $F$-basis of $D$.
The subfield $K=F(i)$ is maximal and $\ast$-invariant.
Note that $1,j$ is a right $K$-basis of $D$ and 
\[
A = \left[ \begin{array}{cc}f(1^\ast  1) & f(1^\ast  j) \\ f(j^\ast  1) & f(j^\ast  j)\end{array} \right]
= \left[ \begin{array}{cc} 1 & 0 \\ 0 & b \end{array} \right] 
= \left[ \begin{array}{cc} a_1 & 0 \\ 0 & a_2 \end{array} \right].
\]
With respect to this basis, the imbedding $\lambda \colon D \to M_2(K)$ is given by
\[
\lambda(z)=\left[ \begin{array}{cc} \alpha+\beta i & b(\gamma+\delta i) \\ 
\gamma - \delta i& \alpha-\beta i \end{array} \right].
\]
The involution $\#$ on $M_2(K)$ which extends the involution $\ast$ of $D$ is given by 
$X^\# = A^{-1} X^\ast A$ where $X^\ast$ is the hermitian transpose of $X$:
\[
\left[ \begin{array}{cc}
x & y \\ u & v 
\end{array} \right]^\# =
\left[ \begin{array}{cc}
x^\ast & b u^\ast \\ b^{-1} y^\ast & v^\ast 
\end{array} \right].
\]
\end{ex}

More general examples (crossed products) will be given later.

\section{Formally real involutions on crossed products}
\label{crossed}

Let $D$ be a central division $F$-algebra with involution $\ast$
and $K$ a maximal subfield of $D$ such that $K^\ast \subseteq K$.
In this section we assume that $K/F$ is a Galois extension and we will
write $G$ for its Galois group. In Chapter 4 of \cite{hers} it is shown that 
there exists a normalized cocycle $\Phi \colon G \times G \to K \setminus \{0\}$
such that $D$ is isomorphic to the crossed product algebra $(K/F,\Phi)$.
By definition, $(K/F,\Phi)$ is a right $K$-vector space with basis
$(e_\sigma)_{\sigma \in G}$ and its multiplication is defined by
\[
\mu \left( \sum_{\sigma \in G} e_\sigma c_\sigma, \sum_{\tau \in G} e_\tau d_\tau \right)
= \sum_{\sigma, \tau \in G} e_{\sigma \tau} \Phi(\sigma,\tau) c_\sigma^\tau d_\tau.
\]
The mapping $f \colon D \to K$ from Section \ref{secinvo} satisfies 
\[
f(\sum_{\sigma \in G} e_\sigma c_\sigma)=c_{\id}.
\]
Namely, for every $\sigma \ne \id$, there exists $k \in K$ such that $k^\sigma \ne k$.
Since $k e_\sigma = e_\sigma k^\sigma$ and $f$ is $K$-$K$ bilinear, it follows that
$f(e_\sigma) k =k^\sigma f(e_\sigma)$. By the choice of $k$, $f(e_\sigma)=0$.
On the other hand $f(e_{\id}) = f(1) =1$.

Our first example (motivated by \cite{ms}) shows that formal reality of $\ast$ does not necessarily imply
formal reality of its extension to $D \otimes_F K$. 

\begin{ex} 
\label{mainex}
Let $F=\CC(a,b)$ be the field of all complex rational functions in two variables
and let $\epsilon = \frac{-1+i\sqrt{3}}{2}$. Let $D_3$ be the symbol algebra
$A_\epsilon(a,b;F)$, i.e. $D_3$ is an $F$-algebra with two generators $x$ and $y$ 
which satisfy the following relations $x^3=a$, $y^3=b$, $yx=\epsilon xy$. 
Let $\ast$ be the involution on $D_3$ which fixes $a,b,x,y$ and conjugates
the elements from $\CC$. Note that $K = F(x)$ is a maximal $\ast$-invariant 
subfield of $D_3$. We claim that $D_3$ is formally real but $D_3 \otimes K$ is not.

By eliminating $a$ and $b$ using relations $x^3=a$, $y^3=b$, we see that 
$D_3$ is the skew field of fractions of the Ore domain
$R = \CC \langle x, y \rangle/ (yx-\epsilon xy)$.
Each element from $R$ can be written uniquely as a linear combination of monomials
$x^m y^n$ with complex coefficients. We pick any monomial ordering $<$ and write $\lt(d)$
for the leading term of $d$ with respect to this monomial ordering. If $\lt(d)=cx^my^n$,
then $\lt(d d^\ast) = c \bar{c} \epsilon^{2mn} x^{2m} y^{2n}$. Since $\CC$ is
formally real, it follows that $R$ is formally real as well. Hence, $D_3$ is also formally real
by Proposition 2 in \cite{ci2}.

The involution $\#$ on $D \otimes_F K \cong M_3(K)$ which extends $\ast$ is given by
$X^\# = A^{-1} X^\ast A$ where  $A = [f(y^i y^j)]_{i,j=0,1,2}$; see Proposition \ref{extend}.
By the discussion above, $f(1)=1$, $f(y)=0$, $f(y^2)=0$, $f(y^3)=b$ and $f(y^4)=0$.
Therefore, $A$ is congruent to the diagonal matrix $\diag(1,b,-b)$. 
Since there is no unital hermitian cone on $S_1(K)$ which contains $1$,$b$ and $-b$, 
$(M_3(K),\#)$ is not formally real by Corollary \ref{firstcorr}.
\end{ex}

The goal of this section is to characterize involutions on $D$
for which the extended involutions on $D \otimes_F K$ are formally real. 
We need an auxiliary result:

\begin{prop}
\label{exti}
If $D=(K/F,\Phi)$ is a division algebra with involution $\ast$
satisfying $K^\ast \subseteq K$, then the following assertions are equivalent:
\begin{enumerate}
\item $(k^\ast)^\sigma=(k^\sigma)^\ast$ for every $k \in K$ and $\sigma \in G$,
\item $e_\sigma^\ast e_\sigma \in K$ for every $\sigma \in G$,
\item $f(e_\tau^\ast e_\sigma)=0$ for every $\sigma, \tau \in G$ such that $\sigma \ne \tau$.
\end{enumerate}
\end{prop}

\begin{proof} For every $k \in K$ and every $\sigma, \tau \in G$, we have that
\[
(k^\tau)^\ast e_\tau^\ast e_\sigma = 
(e_\tau k^\tau)^\ast e_\sigma =
(k e_\tau)^\ast e_\sigma = 
e_\tau^\ast k^\ast e_\sigma=
e_\tau^\ast e_\sigma (k^\ast)^\sigma.
\]
We will use this identity several times.

$(1) \Rightarrow (2)$ Suppose that $(k^\sigma)^\ast= (k^\ast)^\sigma$
for every $k \in K$ and $\sigma \in G$. Then the identity (used with $\tau=\sigma$) implies 
that $e_\sigma^\ast e_\sigma$ commutes with $(k^\ast)^\sigma$ for every $k \in K$, 
hence it commutes with every element from $K$. By the Double Centralizer Theorem, 
it follows that $e_\sigma^\ast e_\sigma \in K$ for every $\sigma \in G$.

$(2) \Rightarrow (1)$ If $e_\sigma^\ast e_\sigma \in K$ for every $\sigma$, 
then the identity implies that $(k^\ast)^\sigma=(k^\sigma)^\ast$ 
for every $k \in K$ and $\sigma \in G$.

$(1) \Rightarrow (3)$ Pick any $\sigma, \tau \in G$ such that $\sigma \ne \tau$.
Then there exists $k \in K$ such that $(k^\tau)^\ast \ne (k^\sigma)^\ast$.
The identity and the assumption imply that
$(k^\tau)^\ast f(e_\tau^\ast e_\sigma) = 
f(e_\tau^\ast e_\sigma) (k^\ast)^\sigma =
f(e_\tau^\ast e_\sigma) (k^\sigma)^\ast$,
hence $f(e_\tau^\ast e_\sigma) = 0$. 

$(3) \Rightarrow (1)$ Suppose that $f(e_\tau^\ast e_\sigma)=0$ 
for every $\sigma, \tau \in G$ such that $\sigma \ne \tau$. The fact that
the matrix $A = [f(e_\tau^\ast e_\sigma)]$ is nonsingular implies that
$f(e_\sigma^\ast e_\sigma) \ne 0$ for every $\sigma \in G$. 
Replacing $\tau$ by $\sigma$ in the identity and applying $f$, we get that 
$(k^\sigma)^\ast f(e_\sigma^\ast e_\sigma)
= f(e_\sigma^\ast e_\sigma) (k^\ast)^\sigma$. 
Since $f(e_\sigma^\ast e_\sigma) \ne 0$, it follows that 
$(k^\sigma)^\ast= (k^\ast)^\sigma$
for every $k \in K$ and $\sigma \in G$. 
\end{proof}

\begin{thm}
\label{crossinvo}
If $D=(K/F,\Phi)$ is a division algebra with involution $\ast$
satisfying $K^\ast \subseteq K$, then the following assertions are equivalent:
\begin{enumerate}
\item $D \otimes K$ is formally real,
\item $D$ is formally real and $(k^\ast)^\sigma=(k^\sigma)^\ast$ for every $k \in K$ and $\sigma \in G$,
\item $e_\sigma^\ast e_\sigma \in K$ for every $\sigma \in G$ and there exists a 
unital hermitian cone on $K$ which contains all of them.
\end{enumerate}
\end{thm}

\begin{proof}
$(1) \Rightarrow (2)$ Suppose that $D \otimes K$ is formally real. 
We claim that $f(e_\sigma^\ast e_\sigma) \ne 0$ for every $\sigma$.
The identity from the proof of Proposition \ref{exti} (used with $\tau=\sigma$)
then implies that $(k^\sigma)^\ast= (k^\ast)^\sigma$
for every $k \in K$ and $\sigma \in G$. 
If the claim is false then the matrix $A$ has a zero entry on the diagonal,
hence it is congruent to a diagonal matrix which has two nonzero entries 
of opposite signs. Hence $A$ cannot belong to a unital hermitian cone on
$(M_n(K),\ast)$. Therefore $(D,\ast) \cong (M_n(K),\#)$ does not
have a unital hermitian cone, contrary to the assumption.
The second part of assertion $(2)$
follows from the fact that $D$ is contained in $D \otimes K$. 

$(2) \Rightarrow (3)$ If $(2)$ is true, then by Proposition \ref{exti}
$e_\sigma^\ast e_\sigma \in K$ for every $\sigma \in G$.
Since $D$ is formally real, the set of all sums of hermitian squares
is a unital hermitian cone which contains $e_\sigma^\ast e_\sigma$ 
and it restricts to a unital hermitian cone on $K$.

$(3) \Rightarrow (1)$ If $(3)$ is true, then by the proposition
$A = [f(e_\tau^\ast e_\sigma)]$ is a diagonal matrix with 
entries $a_\sigma = e_\sigma^\ast e_\sigma$. By the assumption
there exists a unital hermitian cone on $K$ containing all $a_\sigma$.
Hence $D \otimes K$ is formally real by Corollary \ref{firstcorr}.
\end{proof}

For the sake of completeness we also note that $(D,\ast)$ need not be formally real
even if all its maximal $\ast$-subfields are formally real.

\begin{ex} 
\label{mainexvar}
We consider a variant of Example \ref{mainex}. Let $\epsilon = \frac{-1+i\sqrt{3}}{2}$
and $a=b=2$. Let $D$ be a $\QQ(\eps)$-algebra generated with two generators $x, y$ 
and three relations $x^3=a$, $y^3=b$, $yx=\epsilon xy$. 
The involution is defined by $\eps^\ast=\eps^{-1},x^\ast=x,y^\ast=y$. 

We claim that every maximal $\ast$-subfield of $D$ is formally real. Note that 
$\ast|_{\QQ(\eps)} \ne \id$, hence we know by the proof of Lemma \ref{anti}
that every maximal $\ast$-subfield of $D$ can be generated by a symmetric element.
It follows that every maximal $\ast$-subfield can be $\ast$-embedded into
$\CC$ with standard involution, thus it is formally real.

We also claim that $D$ is not formally real. It suffices to see that
\[
d_1^\ast d_1+d_2^\ast d_2+d_3^\ast d_3+d_4^\ast d_4=0,
\]
where
\[
\begin{array}{lll}
d_1 = \eps^{-1} x+x^2+ 2 y, & & d_3 = 2x-x^2+x y^2, \\
d_2 = 1- \eps^{-1} x y - x^2 y^2, & & d_4 = 3-x-x^2.
\end{array}
\]
\end{ex}

\section{Extensions and contractions of unital hermitian cones}

Let $D=(K/F,\Phi)$ be a crossed product division algebra with involution $\ast$
such that $K^\ast \subseteq K$ and $(D \otimes K,\ast) \cong (M_n(K),\#)$ is formally real. 
Let $(e_\phi)_{\phi \in G}$ be the standard right $K$-basis of $D$ and let 
$\lambda \colon D \to M_n(K)$ be the left regular representation of $D$ with respect to 
the standard basis. Let $\mathcal{N}$ be the set of all unital hermitian cones on $(K,\ast)$ 
which contain $a_\sigma = e_\sigma^\ast e_\sigma$ for every $\sigma \in G$. 
By assertion $(3)$ of Theorem \ref{crossinvo}, $\mathcal{N}$ is nonempty.

\begin{lem}
We can define an action of $G$ on $\mathcal{N}$ by 
\[
N_\sigma = \left(\frac{1}{a_\sigma} N \right)^{\sigma^{-1}}
= \{\left(\frac{n}{a_\sigma}\right)^{\sigma^{-1}} \! \vert \ n \in N\}.
\]
\end{lem}
\begin{proof} An element $k \in K$ belongs to $N_\sigma$ if and only if 
$a_\sigma k^\sigma \in N$. It follows that $N_\sigma+N_\sigma \subseteq N_\sigma$ 
and $N_\sigma \cap -N_\sigma=\{0\}$. By assertion $(2)$ of Theorem \ref{crossinvo},
$\ast|_K$ commutes with every element of $G$. It follows that $N_\sigma \subseteq S_1(K)$
and $r^\ast N_\sigma r \subseteq N_\sigma$ for every $r \in K$. For every $\sigma, \tau \in G$
we have $a_\sigma a_\tau^\sigma = e_\sigma^\ast e_\tau^\ast e_\tau e_\sigma =
\Phi(\tau,\sigma)^\ast a_{\tau \sigma} \Phi(\tau,\sigma) \in N$. It follows that $N_\sigma$
contains $a_\tau$ for every $\tau \in G$ and that $(N_\tau)_\sigma=N_{\tau \sigma}$.
\end{proof}

Let $\mathcal{M}$ be the set of all unital hermitian cones on $(D,\ast)$. For every 
$M \in \mathcal{M}$ write $M^c = M \cap K$ and note that $M^c \in \mathcal{N}$.
For every $N \in \mathcal{N}$ write $N^e = \lambda^{-1}\phi^{-1}F(N)=
\{c \in D \vert \ x^\ast A \lambda(c) x \in N \mbox{ for every } x \in K^n \}$
where $A =\diag(a_\sigma)_{\sigma \in G}$. By Corollary \ref{firstcorr}, we
see that $\phi^{-1}F(N)$ is a unital hermitian cone on $(M_n(K),\#)$,
hence $N^e \in \mathcal{M}$.

\begin{thm}
\label{mainext}
Setup from above. For every $N \in \mathcal{N}$ we have 
\begin{enumerate}
\item $N^e =\{u \in D \vert \ f(d^\ast u d) \in N \mbox{ for every } d \in D\}$,
\item $N^{ec} = \bigcap_{\sigma \in G} N_\sigma$,
\item $(N_\sigma)^e = N^e$ for every $\sigma \in G$,
\item $N^{ece} = N^e$. 
\end{enumerate} 
For every $M \in \mathcal{M}$ we have 
\begin{enumerate}
\setcounter{enumi}{4}
\item $(M^c)_\sigma = M^c$ for every $\sigma \in G$,
\item $M^{cec}=M^c$.
\end{enumerate}
\end{thm}

\begin{proof}
To prove $(1)$, pick any $u \in D$ and note that $[e_\sigma^\ast u e_\tau]_{\sigma,\tau \in G}
= [e_\sigma^\ast e_\tau]_{\sigma,\tau \in G} \lambda(u)$. Applying $f$ to all $n^2$ equations,
we get $[f(e_\sigma^\ast u e_\tau)]_{\sigma,\tau \in G}
= [f(e_\sigma^\ast e_\tau)]_{\sigma,\tau \in G} \lambda(u)$ which is equal to $A \lambda(u)$.
For every $x = (k_\phi)_{\phi \in G}$, we get $x^\ast A \lambda(u) x = f(d^\ast u d)$,
where $d = \sum_{\phi \in G} e_\phi k_\phi$. Claim $(1)$ now follows from the definition of $N^e$.

Claim $(3)$ is rather tricky. Pick $\sigma \in G$ and 
$u = \sum_{\omega \in G} e_\omega r_\omega \in D$. 
For every $d=\sum_{\phi \in G} e_\phi k_\phi \in D$ write 
$d_\sigma = \sum_{\phi \in G} e_{\phi \sigma} \Phi(\phi,\sigma) k_\phi^\sigma$. We have
\[
\begin{array}{ccl}
f(d_\sigma^\ast u d_\sigma)  & = & \sum_\omega f(d_\sigma^\ast e_\omega r_\omega d_\sigma) \\
& = & \sum_{\omega,\phi,\tau} f(k_\phi^{\sigma \ast} \Phi(\phi,\sigma)^\ast e_{\phi \sigma}^\ast
e_\omega r_\omega e_{\tau \sigma} \Phi(\tau, \sigma) k_\tau^\sigma) \\
& = & \sum_{\omega,\phi,\tau} f(k_\phi^{\sigma \ast} \Phi(\phi,\sigma)^\ast e_{\phi \sigma}^\ast
e_{\omega \tau \sigma} r_\omega^{\tau \sigma} \Phi(\omega,\tau \sigma) \Phi(\tau,\sigma) k_\tau^\sigma) \\
& \stackrel{(\ast)}{=} & \sum_{\omega,\phi,\tau} k_\phi^{\sigma \ast} \Phi(\phi,\sigma)^\ast a_{\phi \sigma}
\delta_{\phi,\omega \tau} r_\omega^{\tau \sigma} \Phi(\omega\tau, \sigma) \Phi(\omega,\tau)^\sigma
k_\tau^\sigma \\
& = & \sum_{\omega,\tau} k_{\omega \tau}^{\sigma \ast} \Phi(\omega \tau,\sigma)^\ast 
a_{\omega \tau \sigma} 
 r_\omega^{\tau \sigma} \Phi(\omega\tau, \sigma) \Phi(\omega,\tau)^\sigma
k_\tau^\sigma \\ 
&  \stackrel{(\ast\ast)}{=} & \sum_{\omega,\tau} k_{\omega \tau}^{\ast \sigma} 
a_\sigma a_{\omega \tau}^\sigma  r_\omega^{\tau \sigma}  \Phi(\omega,\tau)^\sigma
k_\tau^\sigma \\
& = & a_\sigma \left( \sum_{\omega,\tau} k_{\omega \tau}^{\ast} a_{\omega \tau}
 r_\omega^{\tau}  \Phi(\omega,\tau) k_\tau \right)^\sigma.
\end{array}
\]
At $(\ast)$ we used the cocycle identity and 
$f(e_{\phi \sigma}^\ast e_{\omega \tau \sigma})=a_{\phi \sigma} \delta_{\phi,\omega \tau}$.
At $(\ast\ast)$ we used $k_{\omega \tau}^{\sigma \ast} = k_{\omega \tau}^{\ast \sigma}$ and
$\Phi(\omega \tau,\sigma)^\ast a_{\omega \tau \sigma} \Phi(\omega\tau, \sigma)
= a_\sigma a_{\omega \tau}^\sigma$.
For $\sigma=\id$, we get $f(d^\ast u d) = \sum_{\omega,\tau} k_{\omega \tau}^{\ast} a_{\omega \tau}
 r_\omega^{\tau}  \Phi(\omega,\tau) k_\tau$, hence
\[
\begin{array}{ccl}
f(d_\sigma^\ast u d_\sigma)  
&  = & a_\sigma f(d^\ast u d)^\sigma.
\end{array}
\]
If $u \in N^e$, then $\lambda(d^\ast u d) \in N$ for every $d \in D$. It follows that
$a_\sigma f(d^\ast u d)^\sigma = f(d_\sigma^\ast u d_\sigma)  \in N$ for every $d \in D$.
Hence, $f(d^\ast u d) \in N_\sigma$ for every $d \in D$, which implies that $u \in N_\sigma^e$.
Since every element of $D$ is of the form $d_\sigma$, we can also prove the opposite inclusion.

To prove $(2)$, pick $k \in K$ and note that $k \in N^e$ if and only if 
$\diag(a_\sigma k^\sigma)_{\sigma \in G} = A \lambda(k)$ belongs to $F(N)$.
Hence, $k \in N^e$ if and only if $a_\sigma k^\sigma \in N$ for every $\sigma \in G$.
Claim $(2)$ now follows from the definition of $N_\sigma$. Claim $(4)$ is a simple
consequence of claims $(2)$ and $(3)$.

Claim $(5)$ follows from the fact that an element $k \in K$ belongs to $M^c$
if and only if $a_\sigma k^\sigma = e_\sigma^\ast k e_\sigma \in M^c$. Claim
$(6)$ now follows from Claim $(2)$ applied to $N=M^c$.
\end{proof}

Let us state two simple corollaries.

\begin{cor}
\label{extk}
Setup from Theorem \ref{mainext}. For every unital hermitian cone $N$ on $K$,
the following are equivalent:
\begin{enumerate}
\item $a_\sigma n^\sigma \in N$ for every $n \in N$ and $\sigma \in G$
(i.e. $N \in \mathcal{N}$ and $N_\sigma=N$ for every $\sigma \in \sigma$),
\item $N$ extends to a unital hermitian cone on $D$ (i.e. $N=M \cap K$
for some unital hermitian cone $M$ on $D$),
\item $N$ extends to a unital hermitian cone on $(M_n(K),\#)$ 
(i.e. $N = \lambda^{-1}(L) \cap K$ for some unital hermitian cone $L$ on $(M_n(K),\#)$.)
\end{enumerate}
\end{cor}

\begin{proof} Clearly $(3)$ implies $(2)$. By Claim $(5)$ of Theorem \ref{mainext}, $(2)$ implies $(1)$.
To see that $(1)$ implies $(3)$ take $L=\phi^{-1}F(N)$. By Corollary \ref{firstcorr}, we see that
$L$ is a unital hermitian cone on $(M_n(K),\#)$. Note that $\lambda^{-1}(L) =N^e$ by the definition 
of $N^e$, hence $\lambda^{-1}(L) \cap K=N^{ec}$. By assumption $N=N_\sigma$ for every $\sigma \in G$.
Hence $N^{ec} = \bigcap_{\sigma \in G} N_\sigma = N$ by Claim $(2)$ of Theorem \ref{mainext}.
\end{proof}

\begin{cor}
Setup from Theorem \ref{mainext}. For every unital hermitian cone $M$ on $D$ the following
are equivalent:
\begin{enumerate}
\item $M$ extends to a unital hermitian cone on $(M_n(K),\#)$ (i.e. $M=\lambda^{-1}(L)$
for some unital hermitian cone $L$ on $(M_n(K),\#)$),
\item An element $u$ of $D$ belongs to $M$ if and only if $f(d^\ast u d) \in M$ for every $d \in D$
(i.e. $M=M^{ce}$).
\end{enumerate}
\end{cor}

\begin{proof}
Suppose that $(1)$ is true. By Theorem \ref{arbitrary} and Corollary \ref{firstcorr}, 
there exists a unital hermitian cone $N$ on $K$ containing all $a_\phi, \phi \in G$
such that $L=\phi^{-1}F(N)$. It follows that $M=N^e$ for some $N \in \mathcal{N}$.
Hence $M^{ce}=N^{ece}=N^e=M$ by Claim $(5)$ of Theorem \ref{mainext}.
Now $(2)$ follows from Claim $(1)$ of Theorem \ref{mainext} (with $N=M^c$).

If $(2)$ is true, then $M=M^{ce}$ by Claim $(1)$ of Theorem \ref{mainext} (with $N=M^c$).
You can take $L=\phi^{-1}F(M^c)$ to get $(1)$.
\end{proof}

The following question remains open:
Is $M^{ce} \subseteq M$ for every $M \in \mathcal{M}$?

We finish this section with two examples which prove the following claims:
\begin{enumerate}
\item a unital hermitian cone from $\mathcal{N}$ can have no extension to $D$,
\item a unitial hermitian cone from $\mathcal{N}$ can have two different extensions to $D$,
\item a unital hermitian cone on $D$ can have no extension to $D \otimes_F K$,
\item a unital hermitian cone on $D$ can have two different extension to $D \otimes_F K$.
\end{enumerate}

\begin{ex}
\label{simpex}
Let $D = \left( \frac{a, b}{\QQ} \right)$ where $a,b>0$ and $i^\ast=i,j^\ast=j$.
Then $K=\QQ(i)$ can be identified with $\QQ(\sqrt{a}) \subset \RR$,
where either $\sqrt{a}>0$ or $\sqrt{a}<0$. The ordering of $\RR$ then induces 
unital hermitian cones (in fact orderings) $N_1$ and $N_2$ on $K$, such that 
$i \in N_1$ and $-i \in N_2$. Since $b > 0$, 
$b$ is a sum of four squares, hence both $N_1$ and $N_2$ contain $a_2=b$.
However, neither $N_1$ nor $N_2$ extends to $D$. Suppose that $N_1$ extends to
a unital hermitian cone $M_1$ of $D$. Since $i \in N_1$, it follows that $j^\ast i j \in M_1$,
so that $-bi=j^\ast i j \in M_1 \cap K=N_1$. Since $1/b$ is a sum of squares, 
it follows that $-i \in N_1$, a contradiction. The proof that $N_2$
does not extend to $D$ is analogous. This proves Claim $(1)$ above.
Note that by Corollary \ref{extk} $N_1 \cap N_2$ extends to $D$.

Theorem \ref{arbitrary} and Corollary \ref{firstcorr} tell us that
$L_1 =\phi^{-1}(F(N_1))$ and $L_2 = \phi^{-1}(F(N_2))$ are 
unital hermitian cones on $(M_n(K),\#)$. Since $N_1 \ne N_2$, also $L_1 \ne L_2$.
Assertion $(3)$ of Theorem \ref{mainext} implies that $\lambda^{-1}(L_1)=\lambda^{-1}(L_2)$.
This proves Claim $(4)$ above.
\end{ex}

\begin{ex}
\label{diffex}
Write $F = \RR(a,b)$, $D = \left( \frac{a,b}{F} \right)$ and $K=F(i)$. 
The involution on $D$ is defined by $i^\ast=i$ and $j^\ast =j$.
Let $R$ be the $\RR$-subalgebra of $D$ generated by $i$ and $j$.
Note that $R=\RR \langle i,j \rangle/(ij+ji)$ is an Ore domain
and that its skew field of fractions is $D$. Similarly, the fields of fractions 
of commutative $\RR$-subalgebras $S=\RR[i,j^2]$ and $T=\RR[i^2,j^2]$
are $K$ and $F$ respectively. We will construct two different unital hermitian cones 
$M_1$ and $M_2$ on $R$ such that $M_1 \cap S=M_2 \cap S$. 
By Proposition 2 in \cite{ci2}, $M_1$ and $M_2$ extend uniquely from $R$ to $D$.
These extensions are clearly different, but they have the same restriction to $K$. 
This will prove Claim $(2)$ above.

Every element of $R$ is a linear combination of monomials $i^m j^n$. 
We pick any monomial ordering $<$ and write $\lt(d)$
for the leading term of $d$ with respect to this monomial ordering.
If $s \in R$ is symmetric and $\lt(s)=c i^m j^n$, then $2|mn$.
For every $r$ such that $\lt(r)=u i^k j^l$ we have that
$\lt(r^\ast (-1)^{\frac{mn}{2}} s r)= cu^2 (-1)^{\frac{(m+2k)(n+2l)}{2}} i^{m+2k}j^{n+2l}$.
It follows that 
\[\begin{array}{c}
M_1 = \{s \in \sym(R) \vert \ \lt(s)=ci^m j^n \text{ where } c(-1)^{\frac{mn}{2}} \ge 0\} \text{ and }\\
M_2 = \{s \in \sym(R) \vert \ \lt(s)=ci^m j^n \text{ where } c(-1)^{\frac{mn}{2}+n} \ge 0\}
\end{array}\]
are unital hermitian cones (even Baer orderings) on $R$. Clearly, they are different and they 
have the same restriction to $S$ (take $n$ even).

Note that $j \in M_1$ and $-j \in M_2$.
We will show that there is no unital hermitian cone on $D \otimes_F K$ which contains either 
$\lambda(j)$ or $-\lambda(j)$. Therefore, neither $M_1$ nor $M_2$ extends to $D \otimes_F K$.
This will prove Claim $(3)$ above. Note that 
\[
A \lambda(j) = \left[ \begin{array}{cc} 1 & 0 \\ 0 & b \end{array} \right]
\left[ \begin{array}{cc} 0 & b \\ 1 & 0 \end{array} \right] =
\left[ \begin{array}{cc} 0 & b \\ b & 0 \end{array} \right] \mbox{ is congruent to } 
\left[ \begin{array}{cc} -b & 0 \\ 0 & b \end{array} \right].
\]
Clearly, there is no unital hermitian cone on $K$ which contains $b$ and $-b$. 
Therefore, by Corollary \ref{firstcorr}, there is no unital hermitian cone on 
$(M_n(K),\#)$ which contains $\lambda(j)$. The same proof also works for $-\lambda(j)$.
\end{ex}

\section{Hermitian trace forms}
\label{last}

Let $A$ be a central simple $F$-algebra with involution $\ast$ and 
$\tr \colon A \to F$ its reduced trace. The mapping 
\[
a \mapsto \tr(a^\ast a)
\]
is called the \textit{hermitian trace form} of $(A,\ast)$. Write
$N_{(A,\ast)}$ for the image of this map. We say that the
hermitian trace form is \textit{positive semidefinite} if
$N_{(A,\ast)} \cap - N_{(A,\ast)} = \{0\}$. In this case,
$N_{(A,\ast)}$ is a unital hermitian cone on $F$.

\begin{prop}
Let $D$ be a central division algebra over $F$ with involution $\ast$
and let $K$ be a maximal $\ast$-invariant subfield. 
Consider the following assertions:
\begin{enumerate}
\item There exists a $\ast$-ordering (=multiplicatively closed unital hermitian cone)
on $F$ which contains $N_{(D,\ast)}$.
\item The extension of the hermitian trace form to $(D \otimes_F K,\ast)$ is positive semidefinite.
\item The hermitian trace form on $(D,\ast)$ is positive semidefinite.
\item $(D \otimes K,\ast)$ is formally real, i.e. has a unital hermitian cone.
\item $(D,\ast)$ is formally real.
\end{enumerate}
Then $(1) \Rightarrow (2) \Rightarrow (3) \Rightarrow (4) \Rightarrow (5)$.
\end{prop}

\begin{proof}
Implications $(2) \Rightarrow (3)$ and $(4) \Rightarrow (5)$ are clear.

$(1) \Rightarrow (2)$ Let $P$ be the $\ast$-ordering containing $N = N_{(D,\ast)}$.
Write $M_P = \{c \in \sym(K) \vert \ \forall k \in K \colon \tr(k^\ast c k) \in P\}$. 
To see that $N_{(D \otimes K,\ast)} \cap -N_{(D \otimes K,\ast)}=\{0\}$, it suffices 
to show that $M_P \cap -M_P =\{0\}$ and $N_{(D \otimes K,\ast)} \subseteq M_P$.
If $c \in M_P \cap -M_P$, then $\tr(k^\ast c k)=0$ for every $k \in K$.
If $c \ne 0$, then write $c^{-1}=k_1^\ast k_1-k_2^\ast k_2$ for $k_1=\frac{1 + c^{-1}}{2}$ and 
$k_2=\frac{1 - c^{-1}}{2}$ and note that $\tr(1) = \tr(k_1^\ast c k_1)-\tr(k_2^\ast c k_2)=0$
which is impossible because $\cha F \ne 0$ by the existence of $P$. 
It follows that $M_P$ is a unital hermitian cone on $(K,\ast)$.
Next we show that $N \subseteq M_P$. To see this, pick  $c \in N$ and note that 
$\tr(k^\ast c k) = c \tr(k^\ast k) \in N \cdot N \subseteq P$ for every $k \in K$.
Now we can prove that $N_{(D \otimes K,\ast)} \subseteq M_P$. Pick an $F$-basis
$g_1,\ldots,g_n$ of $D$ such that $\left[ \tr(g_i^\ast g_j) \right]_{i,j}$ is diagonal.
By the definition of $N$, the diagonal elements $\tr(g_i^\ast g_i)$ belong to $N$.
For every element $z \in D \otimes K \cong M_n(K)$ there exist $k_i \in K$
such that $z = \sum_{i=1}^n \lambda(g_i) k_i $. It follows that
$\tr(z^\# z) = \sum_{i=1}^n k_i^\ast k_i \tr(g_i^\ast g_i) \in M_P$,
because $M_P$ is a unital hermitian cone on $K$ which contains $N$.

$(3) \Rightarrow (4)$ Let $N = N_{(D,\ast)}$ and $M_N = 
\{c \in K \vert \ \forall k \in K \colon \tr(k^\ast c k) \in N\}$. 
As above, we see that $M_N \cap -M_N=\{0\}$ hence $M_N$
is a unital hermitian cone on $(K,\ast)$. 
Let $e_i, i=1,\ldots,m$ be a right $K$-basis of $D$ such that 
$\left[ f(e_i^\ast e_j) \right]_{i,j}$ is diagonal and write 
$a_i = f(e_i^\ast e_i)$. Since $\tr(k^\ast a_i k)=
\tr(f((e_i k)^\ast (e_i k)))=\tr((e_i k)^\ast (e_i k)) \in N$
for every $k \in K$, it follows that $a_i \in M_N$ for every $i=1,\ldots,m$.
Now Corollary \ref{firstcorr} implies that $D \otimes K$ is formally real.
\end{proof}

In Section 5 of \cite{ps}, it is proved that the five assertions of
the proposition are equivalent for quaternion algebras $\left( {a,b \over F} \right)$
with standard involution. In Section 4 of \cite{lsu}, it is shown that $(5)$ is equivalent to $(3)$ 
in many other cases. (They use different terminology.) By our Example \ref{mainex},
$(5)$ is not always equivalent to $(4)$. We conjecture that in general any two assertions are inequivalent.

\vskip 5mm

\textbf{Acknowledgements: } I would like to thank to Prof. Tom Craven for sharing with me his experience.
He proved an early version of Theorem \ref{main} (with $\eps=1$ and $\ast=\id$) and read later versions
of the manuscript. I would also like to thank to Dr. Thomas Unger for his comments on Section \ref{last}.

\end{document}